\documentclass[graybox]{SNmult}
\usepackage{type1cm}        
\usepackage{makeidx}         
\usepackage{graphicx}        
\usepackage{multicol}        
\usepackage[bottom]{footmisc}

\usepackage{newtxtext}       %
\usepackage[varvw]{newtxmath}       

\usepackage{url}
\usepackage{amsmath}
\usepackage{booktabs} 
 \usepackage{xcolor}
\spnewtheorem{assumption}{Assumption}{\bfseries}{\itshape}

\makeindex             

\begin{document}
\title*{A semi-smooth Newton method for efficient evaluation of the inverse hysteresis operator}
\titlerunning{A semi-smooth Newton method for the inverse hysteresis operator}
\author{Herbert Egger \orcidID{0000-0003-3769-8791} and\\ Felix Engertsberger\orcidID{0009-0001-6799-8317}}
\institute{Herbert Egger \at Johann Radon Institute for Computational and Applied Mathematics, Austrian Academy of Sciences, Altenbergerstr.~69, 4040~Linz, Austria, \email{herbert.egger@ricam.oeaw.ac.at} \and Felix Engertsberger \at Institute for Numerical Mathematics, Johannes Kepler University, Altenbergerstr.~69, 4040 Linz, Austria \email{felix.engertsberger@jku.at}}

\maketitle

\abstract{We study the numerical evaluation of an energy-based vector hysteresis model and its incorporation into finite element simulations based on a vector potential formulation. The inherent non-smoothness of the hysteresis model poses challenges for both numerical analysis and implementation. Using tools from convex analysis, we characterize the forward and inverse hysteresis operators in terms of energy densities. This characterization yields well-posedness of the resulting models and leads to robust algorithms for their evaluation based on generalized semi-smooth Newton methods. It furthermore enables a seamless integration into magnetic field simulations, leading to nonlinear and non-smooth optimization problems at every load step. We discuss the finite element discretization, present a semi-smooth Newton method for the iterative solution, and establish global linear convergence with mesh-independent convergence rates. The theoretical results are illustrated by numerical experiments.
}

\abstract*{We study the numerical evaluation of an energy-based vector hysteresis model and its incorporation into finite element simulations based on a vector potential formulation. The inherent non-smoothness of the hysteresis model poses challenges for both numerical analysis and implementation. Using tools from convex analysis, we characterize the forward and inverse hysteresis operators in terms of energy densities. This characterization yields well-posedness of the resulting models and leads to robust algorithms for their evaluation based on generalized semi-smooth Newton methods. It furthermore enables a seamless integration into magnetic field simulations, leading to nonlinear and non-smooth optimization problems at every load step. We discuss the finite element discretization, present a semi-smooth Newton method for the iterative solution, and establish global linear convergence with mesh-independent convergence rates. The theoretical results are illustrated by numerical experiments.
}

\section{Introduction}
\label{engertsberger:sec:1}

Hysteresis is recognized as one of the dominant loss mechanisms in ferromagnetic materials and therefore plays a crucial role in the analysis and optimization of high-power electromagnetic devices such as electric machines and transformers. Various hysteresis models and approaches for their incorporation into finite element simulations have been proposed in the literature; see~\cite{engertsberger:Sadowski2002,engertsberger:Upadhaya2020} and the references therein.
We here consider the energy-based hysteresis model proposed by Henrotte et al.~\cite{engertsberger:Lavet2013}, based on earlier work by Bergqvist~\cite{engertsberger:Bergqvist1997}. The model is vectorial, thermodynamically consistent by construction, and well suited for implementation in magnetic scalar potential formulations~\cite{engertsberger:Prigozhin2016,engertsberger:Domenig2024}. The inherent non-smoothness of the model, however, makes both its analysis and its numerical realization rather delicate. These difficulties can be addressed using tools from convex analysis~\cite{engertsberger:Beck2017}. In this work, we study the numerical evaluation of the associated inverse hysteresis operator~\cite{engertsberger:egger2024inverse} and its incorporation into magnetic field simulations based on a vector potential formulation.

\bigskip
\noindent
\textbf{Outline and main contributions.}
We first recall the energy-based hysteresis model, review its main properties, and discuss its efficient implementation. We then introduce the inverse hysteresis operator, summarize the properties established in~\cite{engertsberger:Egger2025semi}, and propose a semi-smooth Newton method~\cite{engertsberger:Ulbrich2011} for its evaluation. This approach resolves numerical difficulties observed in related work~\cite{engertsberger:Jacques2015}. We further discuss the incorporation of the inverse hysteresis model into vector potential formulations of magnetic field problems and their discretization by finite elements. The resulting nonlinear systems are solved by globally convergent iterative methods. The performance of the proposed methods is illustrated using representative benchmark problems.

\section{Energy-based hysteresis model}
\label{engertsberger:sec:2}

As a first step of our analysis, we briefly recall the energy-based hysteresis model proposed by Henrotte et al.~\cite{engertsberger:Lavet2013}. The model is based on the relation
\begin{align}     
\label{engertsberger:eq:1}
\mathbf{B} = \mu_0 \mathbf{H} + \sum\nolimits_k \mathbf{J}_k
\end{align}
between magnetic flux $\mathbf{B}$, magnetic field $\mathbf{H}$, and partial magnetic polarizations $\mathbf{J}_k$. 
The latter are defined implicitly via an incremental energy-dissipation principle
\begin{align} \label{engertsberger:eq:2}
    \mathbf{J}_k = \operatorname{arg\,min}_\mathbf{J} \{ U_k(\mathbf{J}) - \langle \mathbf{J}, \mathbf{H} \rangle + \chi_k |\mathbf{J} -  \mathbf{J}_{k,p}| \},
\end{align}
which immediately guarantees thermodynamic consistency of the model.
The internal energy densities $U_k(\cdot)$ and the pinning coefficients $\chi_k$ are material-dependent quantities, while $\mathbf{J}_{k,p}$ denotes the partial polarization from the previous time step encoding the memory of the system. 
Together with \eqref{engertsberger:eq:1}, this yields the constitutive relation
\begin{align}
    \label{engertsberger:eq:3}
    \mathbf{B}(\mathbf{H};\mathbf{J}_p) = \mu_0 \mathbf{H} + \sum\nolimits_k \mathbf{J}_k(\mathbf{H};\mathbf{J}_{k,p}),
\end{align}
where $\mathbf{J}_k(\mathbf{H};\mathbf{J}_{k,p})$ denotes the solution of \eqref{engertsberger:eq:2}, and $\mathbf{J}_p=\{\mathbf{J}_{k,p}\}$ abbreviates the collection of partial polarizations from the previous load step. This relation will be called the \emph{forward hysteresis operator}.
We impose the following assumptions~\cite{engertsberger:Prigozhin2016}.
\begin{assumption}
    \label{engertsberger:ass}
    Let $\mathbf{J}_{k,p} \in \mathbb{R}^3$, $d=3$, $\chi_k \ge 0$, and the internal energy densities have the form $U_k(\mathbf{J}) = -\frac{2 A_s J_{s,k}}{\pi} \log(\cos(\frac{\pi}{2} \frac{|\mathbf{J}|}{J_{s,k}} ) )$ with positive constants $A_s$ and $J_{s,k} > 0$.
\end{assumption}
The functions $U_k$ are continuously differentiable and $\sigma_k$-strongly convex with constant $\sigma_k = \frac{A_s \pi}{2 J_{s,k}}$. 
This allows to establish the following important properties.
\begin{lemma}  \label{engertsberger:lem:1}
\label{engertsberger:theo:cow}
(i) The problems \eqref{engertsberger:eq:2} have unique minimizers and the forward hysteresis operator \eqref{engertsberger:eq:3} can be expressed as $\mathbf{B}(\mathbf{H};\mathbf{J}_p) = \nabla_{\mathbf{H}} w^*(\mathbf{H};\mathbf{J}_p)$ with co-energy density 
\begin{align}
    \label{engertsberger:eq:cow}
    w^*(\mathbf{H};\mathbf{J}_p) = \frac{\mu_0}{2} |\mathbf{H}|^2 - \sum\nolimits_k \min\nolimits_{\mathbf{J}} \Big( U_k(\mathbf{J}) - \langle \mathbf{H}, \mathbf{J} \rangle + \chi_k |\mathbf{J} - \mathbf{J}_{k,p}| \Big).
\end{align}
(ii) The gradients $\nabla_{\mathbf{H}} w^*(\mathbf{H};\mathbf{J}_p)$ are strongly monotone and Lipschitz continuous, i.e., 
\begin{align}
\langle \nabla_\mathbf{H} w^*(\mathbf{H};\mathbf{J}_p)
- \nabla_\mathbf{H} w^*(\mathbf{H}';\mathbf{J}_p),
\mathbf{H} - \mathbf{H}'\rangle &\ge c_1 \,|\mathbf{H} - \mathbf{H}'|^2 \\
|\nabla_\mathbf{H} w^*(\mathbf{H};\mathbf{J}_p) - \nabla_\mathbf{H} w^*(\mathbf{H}';\mathbf{J}_p)| &\le c_2\, |\mathbf{H}-\mathbf{H}'|
\end{align}
with constants $c_1 = \mu_0$ and $c_2 = \mu_0 + \frac{2 \sum\nolimits_k J_{s,k}}{A_s \pi}$ independent of $\mathbf{H}$ and $\mathbf{J}_p$. \\
(iii) The mapping $\mathbf{B}(\mathbf{H};\mathbf{J}_p) = \nabla_\mathbf{H} w^*(\mathbf{H};\mathbf{J}_p)$ is semi-smooth as a function of $\mathbf{H}$ and an element in the generalized Jacobian $\partial_\mathbf{H} \mathbf{B}(\mathbf{H};\mathbf{J}_p)$ is given by $\mathbb{S}_\mathbf{B} = \mu_0 I + \sum\nolimits_k \mathbb{S}_{\mathbf{J}}^k$ with 
\begin{align}
    \label{engertsberger:eq:slant:j}
    \mathbb{S}_{\mathbf{J}}^k = \begin{cases}
    0, & \text{ if } \mathbf{J}_k = \mathbf{J}_{k,p} \\
    \Big( \nabla_{\mathbf{J} \mathbf{J}}^2 U_k + \frac{\chi_k}{|\mathbf{J}_k - \mathbf{J}_{k,p}|} \Big( I - \frac{(\mathbf{J}_k - \mathbf{J}_{k,p}) \,\otimes\, (\mathbf{J}_k - \mathbf{J}_{k,p})}{|\mathbf{J}_k - \mathbf{J}_{k,p}|^2}\Big) \Big)^{-1}, & \text{ if } \mathbf{J}_k \neq \mathbf{J}_{k,p}
    \end{cases}
\end{align}
\end{lemma}
\noindent 
For details on the terminology and a complete proof of the assertions, we refer to \cite{engertsberger:Beck2017, engertsberger:Egger2025semi}. 

Note that the hysteresis operator $\mathbf{B}(\mathbf{H};\mathbf{J}_p)$ is \emph{not differentiable} in a classical sense. This is a consequence of the non-smooth term in the minimization problem \eqref{engertsberger:eq:2} and represents a characteristic feature of rate-independent hysteresis models~\cite{engertsberger:Mielke2015}. The following observation simplifies the numerical solution of \eqref{engertsberger:eq:2}; see \cite{engertsberger:Prigozhin2016,engertsberger:egger2024inverse} for further details. 
\begin{lemma} \label{engertsberger:lem:2}
The solutions of \eqref{engertsberger:eq:2} are given by $\mathbf{J}_k=\nabla U_k^*(\mathbf{Y}_{k})$ with 
\begin{align}
    \label{engertsberger:eq:forward:dual}
    \mathbf{Y}_k = \operatorname{arg\,min}_{|\mathbf{H} - \mathbf{Y}| \leq \chi_k} \{ U^*_k(\mathbf{Y}) - \langle \mathbf{Y}, \mathbf{J}_{k,p} \rangle \}
\end{align}
and conjugate functions $U_k^*(\mathbf{Y})= \frac{2 J_{s,k}}{\pi} \Big( |\mathbf{Y}| \arctan(\tfrac{|\mathbf{Y}|}{A_s}) - \frac{1}{2} A_s \log(A_s^2 + |\mathbf{Y}|^2) \Big)$. 
\end{lemma}
\begin{remark}
The minimization problems \eqref{engertsberger:eq:forward:dual} are convex and their unique solution can be computed as follows: If $\mathbf{Y}_{k} = \nabla_\mathbf{J} U_k(\mathbf{J}_{k,p})$  fulfills the constraint, then $\mathbf{J}_k = \mathbf{J}_{k,p}$. Otherwise the solution lies on the boundary $|\mathbf{H} - \mathbf{Y}| = \chi_k$ and can be determined efficiently by Newton-type methods; see \cite{engertsberger:Nocedal2006} and \cite{engertsberger:Prigozhin2016}. 
Once the quantities $\mathbf{J}_k$ have been determined, the co-energy density $w^*(\mathbf{H};\mathbf{J}_p)$, the hysteresis operator $\mathbf{B}(\mathbf{H};\mathbf{J}_p)$, and the generalized Jacobians $\mathbb{S}_\mathbf{B}$ can be evaluated efficiently using the formulas of Lemma~\ref{engertsberger:lem:1}.
\end{remark}

\section{Inverse hysteresis operator}
\label{engertsberger:sec:3}

The results of Lemma~\ref{engertsberger:lem:1} imply 
that the mapping $\mathbf{H} \mapsto \mathbf{B}(\mathbf{H};\mathbf{J}_p)$ is invertible. The inverse mapping $\mathbf{B} \mapsto \mathbf{H}(\mathbf{B};\mathbf{J}_p)$ is called the \emph{inverse hysteresis operator}. It is well-defined and inherits important properties from the forward operator. 
In particular, it can be characterized as the gradient of the magnetic energy density, defined by
\begin{align}
    \label{engertsberger:eq:w}
    w(\mathbf{B};\mathbf{J}_p) = \sup\nolimits_{\,\mathbf{H}} \{ \langle \mathbf{B}, \mathbf{H} \rangle  - w^*(\mathbf{H};\mathbf{J}_p) \}.
\end{align}
Note that the functions $w(\mathbf{B};\mathbf{J}_p)$ and $w^*(\mathbf{H};\mathbf{J}_p)$ are related by convex duality \cite{engertsberger:Beck2017}, which is the key ingredient for proving the following properties. 

\begin{lemma} \label{engertsberger:lem:3}
Under Assumption~\ref{engertsberger:ass}, the following assertions hold true: \\
(i) The function \eqref{engertsberger:eq:w} is well-defined and continuously differentiable. Its gradient is
\begin{align}
    \label{engertsberger:eq:mat:w}
    \nabla_\mathbf{B} w(\mathbf{B};\mathbf{J}_p) = \mathbf{H},
\end{align}
where $\mathbf{H}$ is the optimal point in the maximization problem \eqref{engertsberger:eq:w}. \\
(ii) The mapping $\nabla_{\mathbf{B}} w(\mathbf{B};\mathbf{J}_p)$ is the inverse of $\nabla_\mathbf{H} w^*(\mathbf{H};\mathbf{J}_p)$ and satisfies 
\begin{align}
\langle \nabla_\mathbf{B} w(\mathbf{B}) 
- \nabla_\mathbf{B} w(\mathbf{B}'),
\mathbf{B} - \mathbf{B}'\rangle &\ge \tfrac{1}{c_2} |\mathbf{B} - \mathbf{B}'|^2 \\
|\nabla_\mathbf{B} w(\mathbf{B}) - \nabla_\mathbf{B} w(\mathbf{B}')| &\le \tfrac{1}{c_1} |\mathbf{B}-\mathbf{B}'|
\end{align}
with the same constants $c_1, c_2$ as introduced in Lemma~ \ref{engertsberger:lem:1} above. \\
(iii) The operator $\mathbf{H}(\mathbf{B};\mathbf{J}_p) = \nabla_\mathbf{B} w(\mathbf{B};\mathbf{J}_p)$ is semi-smooth as a function of $\mathbf{B}$ and an element of the generalized Jacobian $\partial_\mathbf{B} \mathbf{H}(\mathbf{B};\mathbf{J}_p)$ is given by $\mathbb{S}_\mathbf{H} = (\mathbb{S}_\mathbf{B})^{-1}$. 
\end{lemma}
The assertions follow from standard properties about conjugate functions; see \cite{engertsberger:Beck2017,engertsberger:Gowda2004}. 
A detailed proof can be found in \cite{engertsberger:Egger2025semi}. 
We next discuss the efficient evaluation of \eqref{engertsberger:eq:w}.
To do so, we consider the equivalent minimization problem 
\begin{align} \label{engertsberger:eq:minH}
 \min\nolimits_{\mathbf{H}} f(\mathbf{H}) \quad \text{with} \quad f(\mathbf{H}) = w^*(\mathbf{H};\mathbf{J}_p) - \langle \mathbf{B},\mathbf{H}\rangle.
\end{align}
For minimization of $f(\mathbf{H})$, we employ a line-search method of the form
\begin{align} \label{engertsberger:iter:1}
\mathbf{H}^{n+1} = \mathbf{H}^n + \alpha^n \delta \mathbf{H}^n, \qquad n \ge 0.
\end{align}
The search directions $\delta \mathbf{H}^n$ are chosen as generalized gradient descent directions 
\begin{align} \label{engertsberger:iter:2}
\delta \mathbf{H}^n = -\mathbb{S}^n \nabla f(\mathbf{H}^n)
\end{align}
where $\mathbb{S}^n$ is an appropriate scaling matrix. For determination of the step size $\alpha^n$ we utilize \emph{Armijo backtracking}, i.e., 
\begin{align} \label{engertsberger:iter:3}
\alpha^n = \max\{&\alpha= \rho^m: 
\phi(\alpha) \le \phi(0) + \sigma \alpha \phi'(0), \ m \ge 0\}, 
\end{align}
with merit function $\phi(\alpha) = f(\mathbf{H}^n + \alpha \delta \mathbf{H}^n)$ and parameters $0<\sigma< \frac{1}{2}$ and $\rho<1$. 
\begin{theorem} \label{engertsberger:thm:main}
Let Assumption \ref{engertsberger:ass} hold and the matrices $\mathbb{S}^n$ be symmetric and positive definite with eigenvalues bounded by $0 < c_3 \le \lambda(\mathbb{S}^n) \le c_4$. Then for any $\mathbf{H}^0 \in \mathbb{R}^d$ the iterates $\mathbf{H}^n$ defined by \eqref{engertsberger:iter:1}--\eqref{engertsberger:iter:3} converge to the unique minimizer $\mathbf{H}$ of \eqref{engertsberger:eq:minH} at a linear rate and the contraction factor only depends on the constants $c_i$, $i=1,\ldots,4$. 
For the choice $\mathbb{S}^n = \mathbb{S}_\mathbf{H}^n$, see Lemma~\ref{engertsberger:lem:1}(iii), the convergence is locally superlinear. 
\end{theorem}
\begin{proof}
Global linear convergence follows from the results of \cite[Ch.~3]{engertsberger:Nocedal2006} and the properties of the function $f(\mathbf{H})$ resulting from Lemma~\ref{engertsberger:lem:3}. 
For $\mathbb{S}^n = \mathbb{S}_\mathbf{H}^n = (\mathbb{S}_\mathbf{B}^n)^{-1}$, the iteration \eqref{engertsberger:iter:1}--\eqref{engertsberger:iter:3} amounts to a damped \emph{semi-smooth Newton method} and superlinear convergence is guaranteed by the results of \cite{engertsberger:Ulbrich2011}.
\end{proof}

\begin{remark} \label{engertsberger:rem:main}
The update step of the iterative scheme 
can be written equivalently as 
\begin{align} \label{engertsberger:eq:alternative}
\mathbf{H}^{n+1} = \mathbf{H}^n + \alpha^n \mathbb{S}^n \big(\mathbf{B} - \mathbf{B}(\mathbf{H}^n;\mathbf{J}_p) \big). 
\end{align}
This corresponds to a fixed-point iteration for solving $\mathbf{B}(\mathbf{H};\mathbf{J}_p) = \mathbf{B}$, i.e., for inverting the forward hysteresis operator. 
A closely related scheme has been proposed in \cite{engertsberger:Jacques2015},which however did not include a line search and approximated the Jacobian at turning points of the hysteresis curve by finite differences. Convergence of this method can, in general, not be expected and, indeed, numerical convergence difficulties and convergence problems have been reported in \cite{engertsberger:Jacques2015}. 
By combining the monotonicity and semi-smoothness properties of the forward operator stated in Lemma~\ref{engertsberger:lem:1} with standard globalization arguments for Newton-type methods for convex problems~\cite{engertsberger:Nocedal2006}, we obtain a globally convergent method that overcomes all numerical issues provably.
\end{remark}

\section{Incorporation into a vector potential formulation}
\label{engertsberger:sec:4}

The inverse hysteresis operator $\mathbf{H}=\mathbf{H}(\mathbf{B};\mathbf{J}_p) = \nabla_{\mathbf{B}} w(\mathbf{B};\mathbf{J}_p)$ introduced in the previous section is well-suited for magnetic field simulations based on the vector potential formulation. 
For an incremental load step, the corresponding problem can be written as
\begin{align} \label{engertsberger:eq:minvar}
\min_{\mathbf{A} \in \mathbb{V}} F(\mathbf{A}) \quad \text {with} \quad F(\mathbf{A}) =\int_\Omega w(\operatorname{curl} \mathbf{A};\mathbf{J}_p) - \mathbf{j}_s \cdot \mathbf{A} \, dx.
\end{align}
Here $\Omega \subset \mathbb{R}^3$ is the computational domain and $\mathbb{V} \subset H(\operatorname{curl};\Omega)$ is an appropriate subspace encoding the relevant boundary and gauge conditions. 
This formulation is sufficiently general to cover problems in 2d and 3d as well as their finite element discretizations~\cite{engertsberger:Meunier2008,engertsberger:Egger2024ho}.
For the iterative solution of \eqref{engertsberger:eq:minvar}, we use a line-search method 
\begin{align} \label{engertsberger:iter:11} 
\mathbf{A}^{n+1} = \mathbf{A}^n + \alpha^n \delta \mathbf{A}^n, \qquad n \ge 0
\end{align}
with $\delta \mathbf{A}^n \in \mathbb{V}$ determined from the linearized variational problems 
\begin{align} \label{engertsberger:iter:22}
    \int_\Omega \mathbb{S}^n \operatorname{curl} \delta \mathbf{A}^n \cdot \operatorname{curl} \mathbf{v} \, dx = \int_\Omega  \mathbf{j}_s \cdot \mathbf{v} - \nabla_\mathbf{B} w(\operatorname{curl} \mathbf{A}^n) \cdot \operatorname{curl} \mathbf{v}  \, dx \quad \forall \mathbf{v} \in \mathbb{V}.
\end{align}
The step size $\alpha^n$ is again determined by Armijo backtracking, i.e., 
\begin{align} \label{engertsberger:iter:33}
\alpha^n = \max\{&\alpha= \rho^m: 
\Phi(\alpha) \le \Phi(0) + \sigma \alpha \Phi'(0), \ m \ge 0\}, 
\end{align}
with merit function $\Phi(\alpha) = F(\mathbf{A}^n + \alpha \delta \mathbf{A}^n)$.
From the analysis developed in \cite{engertsberger:Egger2024ho}, we can immediately deduce the following convergence results.
\begin{theorem} \label{engertsberger:thm:main2}
Let $\Omega \subset \mathbb{R}^3$ be a bounded Lipschitz domain, $\mathbf{j}_s \in L^2(\Omega)^3$, $\mathbb{V}$ be a closed subspace of $H(\operatorname{curl};\Omega)$ such that $\|\mathbf{v}\|_{L^2(\Omega)} \le c_5 \|\operatorname{curl} \mathbf{v}\|_{L^2(\Omega)}$ for all $\mathbf{v} \in \mathbb{V}$. 
Furthermore let $w(\mathbf{B};\mathbf{J}_p)$ satisfy the conditions of Lemma~\ref{engertsberger:lem:3} and $\mathbb{S}^n$ be symmetric and positive definite with eigenvalues bounded by $c_3,c_4$.  
Then for any $\mathbf{A}^0 \in \mathbb{V}$, 
the iterates $\mathbf{A}^n$ of \eqref{engertsberger:iter:11}--\eqref{engertsberger:iter:33} converge to the unique solution $\mathbf{A} \in \mathbb{V}_\mathbf{g}$ of \eqref{engertsberger:eq:minvar} at a linear rate, i.e., 
\begin{align}
    \|\operatorname{curl} (\mathbf{A}^n - \mathbf{A})\|_{L^2(\Omega)} \leq C\,q^n \|\operatorname{curl} (\mathbf{A}^0 - \mathbf{A})\|_{L^2(\Omega)}  
\end{align}
and the contraction factor $q$ only depends on the constants $c_i$, $i=1,\ldots,5$.  
\end{theorem}

\begin{remark} 
In ferromagnetic materials, we may choose $\mathbb{S}^n = \mathbb{S}_\mathbf{H}^n \in \partial_\mathbf{B} \mathbf{H}(\operatorname{curl}\mathbf{A}^n;\mathbf{J}_p)$, as defined in Lemma~\ref{engertsberger:lem:3}. The iteration then becomes a semi-smooth Newton method for minimizing \eqref{engertsberger:eq:minvar}.
This method exhibits at least global linear convergence on the continuous and discrete level, and the contraction factor $q$ is independent of the discretization level; see \cite{engertsberger:Egger2024ho,engertsberger:Egger2025semi} for numerical results highlighting this mesh-independent convergence.
\end{remark}

\section{Numerical validation}
\label{engertsberger:sec:5}

We now illustrate the theoretical results by numerical experiments for the stable inversion scheme and the magnetic field simulations based on the vector potential formulation. All computations were carried out in \textsc{Matlab} on a laptop equipped with a 13th Gen Intel Core i5-1335U CPU with a clock rate of 1.30\,GHz. To accelerate the evaluation of the material law, the constitutive model was compiled into a \texttt{.mex} function using \textsc{Matlab} Coder with OpenMP-enabled shared-memory parallelization.

\bigskip
\noindent
\textbf{Stable numerical inversion on a material point.}
To demonstrate the robustness and efficiency of the numerical inversion scheme, we first present results for a single material point.
For the experiment, we use $11$ pinning forces fitted to the measurement data reported in \cite{engertsberger:teamproblem32}. The corresponding material parameters are summarized in Table~\ref{engertsberger:table:param}.

As magnetic field excitation we choose $\mathbf{H}^\ell = 700\,( \sin(t^\ell (5/2) \pi),0)$ with $200$ time steps $t^\ell$ uniformly distributed in the interval $[0,1]$. 
Using the forward hysteresis operator \eqref{engertsberger:eq:3}, we compute the corresponding magnetic fluxes $\mathbf{B}^\ell$, starting from a completely demagnetized state $\mathbf{J}_{k,p}(t^0) = 0$. 
The resulting sequence serves as the input for the inverse hysteresis operator \eqref{engertsberger:eq:mat:w}. For the numerical inversion, we employ the method of Lemma~\ref{engertsberger:lem:2}. Iterations are stopped with a relative tolerance of $\text{tol} = 10^{-8}$.
The corresponding results are depicted in Figure \ref{engertsberger:fig:matpoint}. 
\begin{figure}[ht!]
    \centering
    \includegraphics[trim=0.4cm 0cm 0.6cm 0.1cm, clip,width=0.48\linewidth]{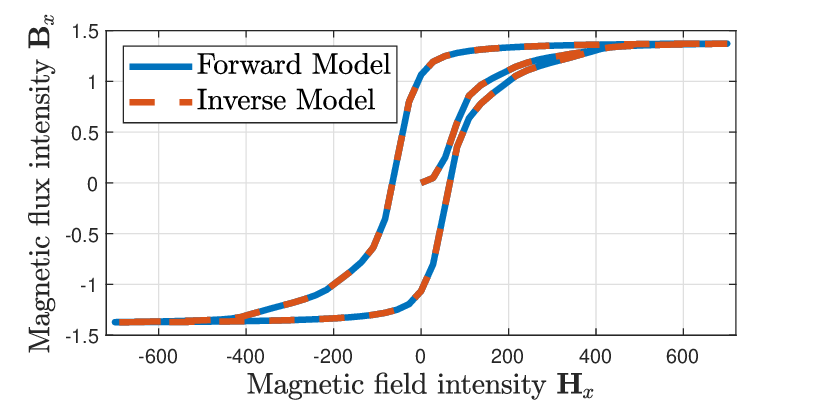}
    \includegraphics[trim=0.3cm 0cm 0.1cm 0.5cm, clip,width=0.51\linewidth]{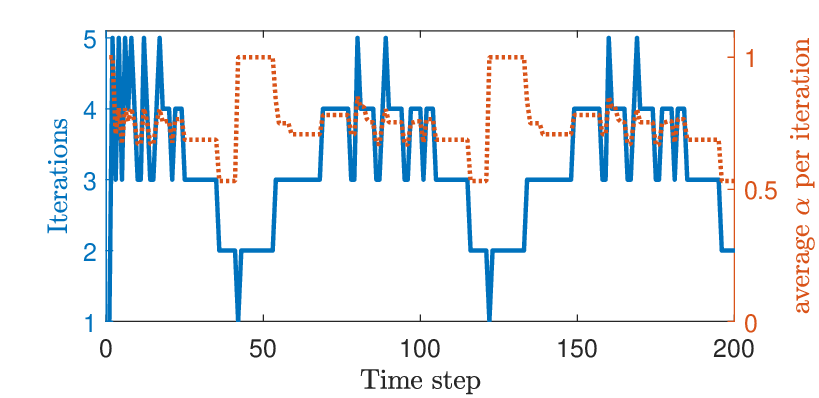}
    \caption{Hysteresis curve of forward and inverse hysteresis operator for an unidirectional excitation (left), iteration number and average step sizes $\alpha$ of the inversion scheme for different time-steps (right)}
    \label{engertsberger:fig:matpoint}
\end{figure}
The left plot clearly illustrates the equivalence of both models. The right plot visualizes the number of iterations and average step sizes required for the individual load steps. Note that only 3--4 iterations are required on average, which illustrates the efficiency of the proposed method. Furthermore, the frequent occurrence of values $\alpha_{\text{avg}} < 1$ highlights the necessity of the line search method.

\begin{table}[ht!]
\centering
\small
\begin{tabular}{@{}cll@{}}
\toprule
\textbf{Parameter} & \textbf{Values} & \textbf{Description} \\
\midrule
$A_s$ &
20 &
field strength \\
$J_{s,k}$ &
0.08, 0.55, 0.49, 0.05, 0.04, 0.07, 0.02, 0.02, 0.02, 0.02, 0.03 &
reference polarization \\
$\chi_k$ &
0, 43, 73, 128, 164, 202, 241, 280, 320, 360, 400 &
pinning strength \\
\bottomrule
\end{tabular}
\caption{Material data for the energy-based hysteresis model used in both numerical tests.}
\label{engertsberger:table:param}
\end{table}

\bigskip
\noindent 
\textbf{Magnetic field simulation. }
Next we present numerical results for a finite element simulation of the three-limbed transformer of the TEAM 32 benchmark problem \cite{engertsberger:teamproblem32}. The specific setup allows the consideration of a quasi-2d setting; the relevant two-dimensional cross-section $\Omega^{2d}$ of the domain $\Omega$ is depicted in the left plot of Fig.~\ref{engertsberger:fig1}.
\begin{figure}[ht!]
    \centering
    \includegraphics[trim=1.6cm 2cm 0.6cm 2cm, clip,width=0.46\linewidth]{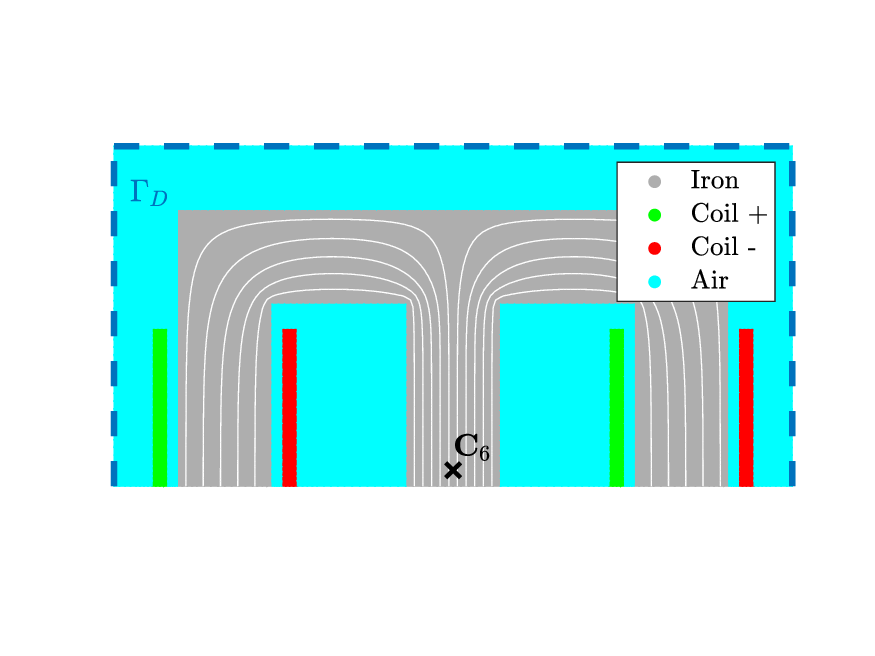}
    \includegraphics[trim=0.0cm 0.1cm 1cm 0.1cm, clip,width=0.53\linewidth]{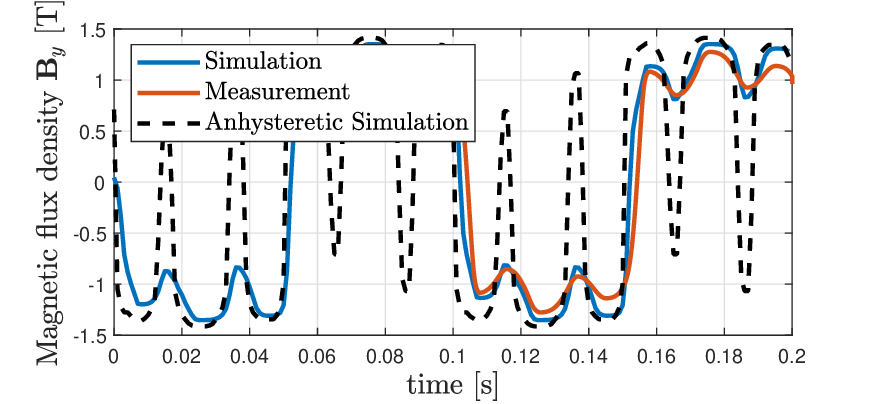}
    \caption{Two-dimensional geometry of a three-limbed transformer for the TEAM 32 benchmark problem (left), and comparison of the simulated magnetic induction at the measurement point C6 to the measurements (right)}
    \label{engertsberger:fig1}
\end{figure}
In the air and the coils we consider a quadratic energy density $w(\mathbf{B}) = \frac{\nu_0}{2} |\mathbf{B}|^2$ while in the ferromagnetic core we use the hysteretic energy $w(\mathbf{B};\mathbf{J}_p)$ with the material data from Table \ref{engertsberger:table:param}. 
For the discretization of the relevant component $\mathbf{A}_z$, we use piecewise linear finite elements defined on a triangular mesh of the 2d cross section $\Omega^{2d}$. 
After every loading step the previous values $\mathbf{J}_{k,p}$ at every quadrature point are updated. In the coils we impose a sinusoidal current with some higher harmonics, which corresponds to the CASE 2 configuration of the benchmark problem. Finally, we terminate the algorithm if the difference in energy is smaller than a relative tolerance of $\text{tol}=10^{-7}$.

The magnetic induction at the pick-up coil C6 is depicted in the right plot of Fig.~\ref{engertsberger:fig1} and compared to measurements and an anhysteretic
simulation.
In Table~\ref{engertsberger:tab:field}, we summarize the convergence behavior for different mesh resolutions.
\begin{table}[ht!]
\centering
\begin{tabular}{c c c c c}
\toprule
\textbf{\# DoF} &
\textbf{\# Elements} &
\textbf{Time } &
\textbf{Iterations} &
\textbf{Iron losses } \\
\midrule
369   &   661   &  2.01 s & 5.04 it & 8.29 J \\
1,369  &  2,644   &  5.72 s & 5.37 it & 8.28 J\\
5,439  & 10,576   &  24.3 s & 5.73 it&  8.29 J\\
21,453 & 42,304   & 117 s & 5.97 it & 8.30 J\\
85,209 & 169,216   & 788 s & 6.12 it & 8.31 J\\
\bottomrule
\end{tabular}
\caption{Total computation times, average iteration counts, and iron losses for simulation of transformer benchmark with 200 time steps and varying discretization levels.}
\label{engertsberger:tab:field}
\end{table}
The nearly constant average iteration counts for varying mesh sizes clearly illustrate the mesh-independent convergence behavior of the proposed method.
To compute the total losses the power density $p_H = \frac{1}{\Delta t}\sum\nolimits_k \chi_k |\mathbf{J}_k - \mathbf{J}_{k,p}|$ is integrated in time and over the whole domain $\Omega$, and then multiplied by the length $L = 2.5$ mm of the transformer. The corresponding results of our computations are summarized in Table~\ref{engertsberger:tab:field}.

\begin{acknowledgement}
This work was supported by the joint DFG/FWF Collaborative Research Centre CREATOR (DFG: Project-ID 492661287/TRR 361; FWF: 10.55776/F90).
\end{acknowledgement}

\end{document}